\pdfoutput=1

\documentclass[12pt]{scrartcl}
\usepackage[utf8]{inputenc}
\usepackage[T1]{fontenc}
\usepackage{fixltx2e}
\usepackage{lmodern}
\usepackage{authblk}
\usepackage{ccfonts}
\usepackage[left=2.5cm,top=2cm,right=2cm,bottom=2cm,a4paper]{geometry}
\usepackage{graphicx}
\usepackage{mathptmx}      
\usepackage{mathtools}
\usepackage{pgf}
\usepackage{amsmath,amsthm,amssymb,cancel,units}
\newtheorem{theorem}{Theorem}[section]

\newtheorem{lemma}[theorem]{Lemma}

\newcommand\R{\mathbb{R}}
\newcommand\calM{\mathcal{M}}
\newcommand\calR{\mathcal{R}}

\newcommand\eps{\varepsilon}
\newcommand\h{h}

\newcommand*{\dt}[1]{\overset{\mbox{\large\bfseries.}}{#1}}

\begin{document}

\title{A projector-splitting integrator for dynamical low-rank approximation
\thanks{This work was partially supported by DFG, SPP 1324,
by RFBR grants 12-01-00546-a,  11-01-00549-a, 12-01-33013 mol-ved-a,
RFBR-DFG grant 12-01-91333, by Federal program ``Scientific and
scientific-pedagogical personnel of innovative
Russia''(contracts 14.740.11.1067, 16.740.12.0727, grants 8500 and 8235)
}}

\author[*]{Ch.~Lubich}
\author[**]{I.V.~Oseledets}
\affil[*]{Mathematisches Institut, Universit\"at T\"ubingen, Auf der Morgenstelle 10, D-72076 T\"ubingen, Germany, \texttt{lubich@na.uni-tuebingen.de}}
\affil[**]{Institute of Numerical Mathematics, Russian Academy of Sciences, Gubkina Street 8, Moscow,Russia, \texttt{ivan.oseledets@gmail.com}}



\maketitle

\begin{abstract}
The dynamical low-rank approximation of time-dependent matrices is a low-rank factorization updating technique. It leads to differential equations for factors of the matrices, which need to be solved numerically. We propose and analyze a fully explicit, computationally inexpensive integrator that is based on  splitting the orthogonal projector onto the tangent space of the low-rank manifold. As is shown by  theory and illustrated by numerical experiments, the integrator enjoys robustness properties that are not shared by any standard numerical integrator. This robustness can be exploited to change the rank adaptively. Another application is in optimization algorithms for low-rank matrices where truncation back to the given low rank can be done efficiently by applying a step of the integrator proposed here.
\end{abstract}

\section{Introduction}
\label{sec-1}
Low-rank approximation of large matrices and tensors is a basic model reduction technique in a wide variety of applications ranging from quantum physics to information retrieval. In the present paper, we consider the low-rank approximation of time-dependent matrices  $A(t)$, $t_0\le t \le \bar t$, which are either given explicitly or are the unknown solution of a differential equation $\dt A(t) =F(A(t))$. In the first case, instead of performing an (expensive) low-rank approximation via singular value decompositions independently for every $t$, one would prefer a computationally inexpensive updating procedure that works only with the increments of $A$ and is free from decompositions of large matrices. In the second case, one would like to find an approximate solution to the differential equation working only with its low-rank approximation.

Both cases can be dealt with by {\it dynamical low-rank approximation} where $A(t)$ is approximated by $Y(t)$ of fixed rank $r$ by imposing that its time derivative should satisfy 
\begin{equation}\label{integr:var}
\| \dt Y(t) - \dt A(t) \| = \min,
\end{equation}
where the minimization is over all matrices  that are tangent to $Y(t)$ on the manifold $\calM_r$ of matrices of rank $r$, and the norm is the Frobenius norm. In the case of a differential equation $\dt A=F(A)$, the above minimization is replaced with
\begin{equation}\label{integr:var-ode}
\| \dt Y(t) -F(Y(t)) \| = \min.
\end{equation}
In both cases, this leads to a differential equation on the manifold of matrices of a given rank~$r$, which is then solved numerically.

For large time-dependent matrices, this approach was proposed and analyzed in \cite{KocL07}, and first numerical applications were given in \cite{NonL08}. The approach and its error analysis were extended to tensors in the Tucker format in \cite{KocL10}. Very recently, in \cite{ArnJ12*,LRSV12*,okhs-dyntt-2012}  the dynamical low-rank approximation approach was extended to tensors in the tensor train (TT) format studied in \cite{osel-tt-2011} and the hierarchical Tucker (HT) format studied in \cite{HacK09}. In quantum dynamics, such an approach was used previously in the multiconfiguration time-dependent Hartree (MCTDH) method 
\cite{MeyGW09,MeyMC90} to determine an approximate solution to the time-dependent multi-particle Schr\"odinger equation, and the above minimization process is known there as the {Dirac--Frenkel time-dependent variational principle} (see, e.g., \cite{Lub08}).

The dynamical low-rank approximation leads to differential equations that need to be solved numerically. In this paper we present a numerical integration technique that is fully explicit and, in contrast to standard integrators such as (explicit or implicit) Runge--Kutta methods, does not suffer from a possible ill-conditioning of certain matrices arising in the differential equations. This new method is based on a splitting of the projector onto the tangent space of the low-rank manifold at the current position. A different splitting algorithm was recently proposed in \cite{okhs-dyntt-2012}, where the differential equations are split along different components. The projector splitting discussed here offers, however, several advantages: it leads to a much simpler and less expensive time-stepping algorithm and it can be shown to enjoy remarkable robustness properties under the ill-conditioning mentioned before.

In Section~\ref{sec-2} we recapitulate the dynamical low-rank approximation of matrices, and in Section~\ref{sec-3} we describe the novel splitting integrator. In Section~\ref{sec-robust} we analyze the robustness under over-approximation, that is, under the dreaded ill-conditioning mentioned above.
Section~\ref{sec-numexp} presents numerical experiments. The final section addresses some perspectives for the use of the proposed integrator and extensions.

\section{Dynamical low-rank approximation of matrices}
\label{sec-2}

Let $r$ be a given rank, and let $\calM_r$ denote the manifold of all real $m\times n$ matrices of rank $r$ (typically, $r\ll m,n$). In the dynamical low-rank approximation to matrices $A(t)\in \R^{m\times n}$, the approximation $Y(t)\in \calM_r$ is represented in a non-unique way as
\begin{equation}\label{USV}
Y(t)=U(t)S(t)V(t)^\top,
\end{equation}
where $U(t)\in \R^{m\times r}$ and $V(t)\in \R^{n\times r}$ each have $r$ orthonormal columns, and $S(t)\in \R^{r\times r}$ is an invertible matrix. This looks similar to the singular value decomposition, but $S(t)$ is not assumed diagonal.

It is shown in \cite[Prop.\,2.1]{KocL07} that, given such a decomposition $Y_0=U_0S_0V_0^\top$ of the starting value and imposing the gauge conditions
\begin{equation}\label{integr:gauge}
U(t)^\top\dt U(t) =0,\quad\ V(t)^\top\dt V(t) =0,
\end{equation}
the solution $Y(t)\in \calM_r$ to the time-dependent
variational principle \eqref{integr:var} admits a unique decomposition (\ref{USV}), and the factors $U(t), S(t), V(t)$ satisfy the following system of differential
equations:
\begin{equation}\label{integr:system}
  \begin{split}
  &\dt{U}(t) = (I - U(t) U(t)^{\top})\dt{A}(t) V(t) S(t) ^{-1}\\
  &\dt{V}(t) = (I - V(t)V(t)^{\top})\dt{A}(t)^{\top} U(t) S(t)^{-\top} \\
  &\dt{S}(t) = U(t)^{\top} \dt{A}(t) V(t).
  \end{split} 
\end{equation}
This system of  differential equations is to be solved numerically.
The numerical solution of
\eqref{integr:system} by standard integrators (e.g., of Runge--Kutta type) becomes cumbersome if $S(t)$ is nearly
singular. This situation occurs in the case of \emph{over-approximation}, when
the true rank of the solution (or approximate rank) is smaller than
the chosen rank $r$. This is a realistic case: the effective rank may
not be known in advance, and it is often reasonable to overestimate
the rank for accurate approximation. To avoid the possible singularity
of $S(t)$ usually some sort of regularization is used. However, such
a regularization introduces errors that are poorly understood.

The minimization condition (\ref{integr:var}) states that at an approximation $Y(t)\in \calM_r$, the derivative $\dt Y(t)$ is obtained by orthogonally projecting $\dt A(t)$ onto the tangent space $T_{Y(t)}\calM_r$ of the rank-$r$ manifold at $Y(t)$:
\begin{equation}\label{integr:sys2}
 \dt{Y}(t) = P(Y(t)) \dt{A}(t),
\end{equation}
where $P(Y)$ is the orthogonal projector onto the tangent space $T_Y\calM_r$  of the manifold $\calM_r$ at $Y\in \calM_r$.
The projector has a  simple representation \cite[Lemma 4.1]{KocL07}: for $Y=USV^\top$ as in~(\ref{USV}),
\begin{equation}\label{integr:proj}
  P(Y) Z = ZVV^{\top}- UU^\top Z VV^{\top} +UU^\top Z .
\end{equation}
Note that $UU^\top$ is the orthogonal projector onto the range $\calR(Y)$ of $Y=USV^\top$, and $VV^\top$ is the orthogonal projector onto the range $\calR(Y^\top)$, so that we can also write
\begin{equation}\label{integr:proj2}
  P(Y) Z =  ZP_{\calR(Y^\top)} - P_{\calR(Y)} Z P_{\calR(Y^\top)} +P_{\calR(Y)} Z .
\end{equation}

\section{The integrator}
\label{sec-3}

\subsection{First-order splitting method, abstract formulation}
\label{sec-3-1}
Let a rank-$r$ approximation $Y_0$ to $A(t_0)$ be given and consider a step of the standard Lie--Trotter splitting of (\ref{integr:sys2}) with (\ref{integr:proj2}) from $t_0$ to $t_1=t_0+\h$:
\begin{enumerate}
\item Solve the differential equation $\dt Y_I = \dt A P_{\calR(Y_I^\top)}$ with initial value $Y_I(t_0)=Y_0$ on the interval $t_0\le t \le t_1$.
\item
Solve the differential equation $\dt Y_{II} =  -P_{\calR(Y_{II})}\dt A P_{\calR(Y_{II}^\top)}$ with initial value $Y_{II}(t_0)=Y_{I}(t_1)$ on the interval $t_0\le t \le t_1$.
\item Solve the differential equation $\dt Y_{III} =  P_{\calR(Y_{III})}\dt A$ with initial value $Y_{III}(t_0)=Y_{II}(t_1)$ on the interval $t_0\le t \le t_1$.
\end{enumerate}
Finally, take $Y_1=Y_{III}(t_1)$ as an approximation to $Y(t_1)$, the solution of (\ref{integr:sys2}) at $t_1$. By standard theory, this is a method of first-order accuracy.
Remarkably, each of the split differential equations can be solved exactly in a trivial way.

\begin{lemma} \label{lem:split}
The solution of 1. is given by
\begin{equation}\label{YI}
Y_{I}(t) = U_{I}(t)S_{I}(t)V_{I}(t) \quad\hbox{ with }\quad
(U_{I} S_{I})^{\dt{\phantom{.}}}= \dt A V_{I}, \quad \dt V_{I}=0.
\end{equation}
The solution of 2. is given by
\begin{equation}\label{YIII}
Y_{II}(t) = U_{II}(t)S_{II}(t)V_{II}(t) \quad\hbox{ with }\quad
\dt S_{II}= - U_{II}^\top \dt A V_{II}, \quad \dt U_{II}=0, \ \dt V_{II}=0.
\end{equation}
The solution of 3. is given by
\begin{equation}\label{YII}
Y_{III}(t) = U_{III}(t)S_{III}(t)V_{III}(t) \quad\hbox{ with }\quad
(V_{III} S_{III}^\top)^{\dt{\phantom{.}}}= \dt{A}^\top U_{III}, \quad \dt U_{III}=0.
\end{equation}
\end{lemma}

\begin{proof} We first notice that each of the terms on the right-hand side of (\ref{integr:proj2}) is in the tangent space $T_Y\calM_r$, because for the first term the orthonormality $V^\top V=I$ yields
$$
P(Y)(ZVV^\top) = ZVV^\top + UU^\top ZVV^\top - UU^\top ZVV^\top = ZVV^\top,
$$
so that $ZVV^\top \in T_Y\calM_r$. Similarly we also have $UU^\top Z \in T_Y\calM_r$ and $UU^\top ZVV^\top  \in T_Y\calM_r$. It therefore follows that the solutions of 1.-3. all stay of rank $r$. Hence, $Y_I(t)$ can be factorized as $Y_{I}(t) = U_{I}(t)S_{I}(t)V_{I}(t)$ with an invertible $s\times s$ matrix $S_I$ and $U_I,V_I$ having orthonormal columns. All the matrices can be chosen to be differentiable, and we thus have
$$
\dt Y_I = (U_{I} S_{I})^{\dt{\phantom{.}}}V_I^\top + (U_IS_I)\dt V_I^\top
$$
which by 1. must equal $\dt Y_I = \dt  A V_IV_I^\top$. We observe that this is satisfied if $(U_{I} S_{I})^{\dt{\phantom{.}}}=\dt A V_I$ and $\dt V_I=0$. 

The proofs for Steps 2.~and 3.~are similar.
\end{proof}
%

\subsection{First-order splitting method, practical algorithm}
\label{sec-3-2}
Lemma~\ref{lem:split} leads us to the following algorithm. Given a factorization (\ref{USV}) of the rank-$r$ matrix
$Y_0=U_0S_0V_0^\top$ and denoting the increment $\Delta A=A(t_1)-A(t_0)$, one step of the integrator reads as follows:
\begin{enumerate}
\item Set 
\[
   K_1 = U_0S_0 + \Delta A \, V_0
\]
 
and compute the factorization 
\[
     U_1\widehat S_1=K_1
\]
with  $U_1$ having orthonormal columns and an $r\times r$ matrix $\widehat S_1$ (using QR or SVD).
\item Set
\[
     \widetilde S_0 = \widehat S_1 -U^{\top}_1 \,\Delta{A}\, V_0.
\]
\item Set
\[
    L_1=  V_0 \widetilde S_0^\top + \Delta{A}^{\top} U_1
\]
and compute the factorization 
\[
     V_1 S_1^\top = L_1, 
\]
with  $V_1$ having orthonormal columns and an $r\times r$ matrix $S_1$ (using QR or SVD).
\end{enumerate}

The algorithm computes a factorization of the rank-$r$ matrix
\[
     Y_1 = U_1 S_1 V^{\top}_1,
\]
which is taken as an approximation to $Y(t_1)$. Note that $Y_1$ is identical to the result of the abstract splitting algorithm of Section~3.1, without any further approximation.

\subsection{Higher-order schemes}
\label{sec-3-3}

Higher-order extensions can be obtained from the above first-order algorithm
by the standard technique of composition rules. 
The usual symmetric composition is obtained by first taking a step of the above integrator with step size $\h/2$ followed by
taking its steps in reverse order.  The resulting
scheme looks as follows (here $A_0 = A(t_0),
A_{1/2} = A(t_0+ \frac{\h}{2}), A_1 = A(t_0+ \h)$):
\begin{equation}\label{integr:scheme}
\begin{split}
 &K_{1/2} = U_0S_0 + (A_{1/2} - A_0) V_0, \\
 &(U_{1/2}, \widehat S_{1/2}) = \textrm{QR}(K_{1/2}), \\
 & \widetilde S_0 = \widehat S_{1/2} - U_{1/2}^\top (A_{1/2}-A_0)V_0,\\
 & L_1 = V_0 \widetilde S_0^\top + (A_1-A_0)^\top U_{1/2},\\
 & (V_1, \widehat S_1^\top)= \textrm{QR}(L_1),\\
 & \widetilde S_{1/2} = \widehat S_1 - U_{1/2}^T (A_1-A_{1/2}) V_1,\\
 & K_1 = U_{1/2}\widetilde S_{1/2} + (A_1-A_{1/2}) V_1,\\
 & (U_1,S_1) = \textrm{QR}(K_1),\\
 &Y_1 = U_1 S_1 V^{\top}_1.
\end{split}
\end{equation}
This symmetrized splitting is a
second-order scheme for \eqref{integr:sys2}. Higher-order schemes are obtained by suitable further compositions; see, e.g., \cite{HaiLW06,McLQ02}.

\subsection{The integrator for matrix differential equations}
\label{sec-3-4}
The basic first-order scheme extends straightforwardly to an explicit method for the low-rank approximation of solutions $A(t)$ of matrix differential equations $\dt A=F(A)$, where now $A(t)$ is not known beforehand. The only change is that $\Delta A= A(t_1)-A(t_0)$ is replaced, in a way resembling the explicit Euler method, with
$$
\Delta A = \h F(Y_0).
$$
The symmetric composition with the adjoint method now yields an implicit method. An explicit method of order 2 is obtained by first taking a step with the basic first-order integrator, which yields an approximation $\widetilde Y_1$, and then to take a step with the above second-order method in which $\dt A(t)$ is replaced, at $t=t_0+\theta\h$, with the linear function $B(t_0+\theta\h)=(1-\theta) F(Y_0) + \theta F(\widetilde Y_1) $, and correspondingly $A(t)$ with the quadratic function $Y_0 + \int_{t_0}^t B(s)\, ds$, that is,
$$
A(t_0+\theta h) \approx Y_0 + h \int_0^\theta B(t_0+\vartheta\h)\, d\vartheta =
Y_0 + \frac h2 \,\theta(2 - \theta)\, F(Y_0) + \frac h2 \theta^2 F(\widetilde Y_1).
$$
Higher-order methods can again be obtained by composition of steps of the second-order method.

\section{Robustness under over-approximation}
\label{sec-robust}
The equations of motion (\ref{integr:system}) break down when $S$ becomes singular, and standard numerical integrators applied to (\ref{integr:system}) run into problems with stability and accuracy when $S$ is ill-conditioned. Such a situation arises when the matrix $A(t)$ to be approximated by a rank-$r$ matrix has rank less than $r$, or is close to a rank-deficient matrix. It is a remarkable property of the integrator proposed here that it behaves much better than a standard integrator applied to  (\ref{integr:system}) in such a situation of over-approximation. On the one hand, this is already apparent from the observation that there is no matrix inversion in the algorithm. There is in fact more to it.

The following result depends on the ordering of the splitting of the projector \eqref{integr:proj2}, so that we first compute $K=US$, then $S$, then $L=VS^\top$. For a different ordering, such as computing subsequently $K$, $L$, $S$, the following surprising exactness result is not valid.

\begin{theorem} \label{thm:exact}
Suppose that $A(t)$ has rank at most $r$ for all $t$. With the initial value $Y_0=A(t_0)$, the splitting algorithm of Section~\ref{sec-3-2} is then exact: $Y_1=A(t_1)$.
\end{theorem}

\begin{proof}
We decompose 
$
A(t) =  U(t) S(t)  V(t)^\top,
$
where both $ U(t)$ and $ V(t)$  have $r$ orthonormal columns, and $S(t)$ is an $r\times r$ matrix. We assume that $V(t_1)^\top V(t_0)$ is invertible. If this is not satisfied, then we make the following argument with a small perturbation of $A(t_1)$ such that $V(t_1)^\top V(t_0)$ becomes invertible, and let the perturbation tend to zero in the end.

The first substep of the algorithm, starting from $Y_0=U_0S_0V_0^\top=U(t_0)S(t_0)V(t_0)^\top=A(t_0)$, yields
$$
U_1\widehat S_1 = A(t_1) V_0 = U(t_1)S(t_1)\bigl(V(t_1)^\top V(t_0)\bigr),
$$
so that the range of
$$
A_1=A(t_1) = \bigl(U(t_1)S(t_1)\bigr)V(t_1)^\top = U_1\widehat S_1 (V(t_1)^\top V(t_0)\bigr)^{-1}V(t_1)^\top
$$
is contained in the range of $U_1$, and hence we have
$$
U_1U_1^\top A_1=A_1, \quad\hbox{as well as }\ A_0V_0V_0^\top = A_0.
$$
Using the formulas of the splitting scheme we then calculate
\begin{eqnarray*}
 Y_1 &=& U_1S_1V_1^\top 
\\
&=&   U_1 \widetilde S_0 V_0^\top +U_1U_1^\top \Delta A
\\
&=& U_1 \widehat S_1 V_0^\top - U_1U_1^\top \Delta A V_0V_0^\top + U_1U_1^\top \Delta A 
\\
&=& U_0S_0V_0^\top+  \Delta A V_0V_0^\top - U_1U_1^\top \Delta A V_0V_0^\top  + U_1U_1^\top \Delta A 
\\
&=& A_0 +A_1V_0V_0^\top- A_0 - A_1V_0V_0^\top + U_1U_1^\top A_0  + A_1 - U_1U_1^\top A_0  = A_1,
\end{eqnarray*}
which is the stated result.
\end{proof}

Consider now a small perturbation to a matrix of rank less than $r$: with a small parameter~$\eps$, assume that (with primes as notational symbols, not derivatives), 
$$
A(t) = A'(t) + \eps A''(t) \quad\hbox{ with rank$(A'(t))=q<r$},
$$
where $A'$ and $A''$ and their derivatives are bounded independently of $\eps$. We factorize
$$
A'(t) = U'(t)S'(t)V'(t)^\top
$$
with $U'(t)$ and $V'(t)$ having $q$ orthonormal columns and with an invertible $q\times q$ matrix $S'(t)$.

We apply the splitting integrator for the dynamical rank-$r$ approximation of $A(t)$ with starting value
$$
Y_0 = A'(t_0) + \eps A_0'', \qquad \hbox{rank}(Y_0)= r,
$$
where $A_0''$ is bounded independently of $\eps$ (but may differ from $A''(t_0)$).
We compare the result of the rank-$r$ algorithm with that of the rank-$q$ algorithm starting from
$$
\bar Y_0 = A'(t_0) + \eps \bar A_0'', \qquad \hbox{rank}(\bar Y_0)=q<r.
$$
\begin{theorem}
In the above situation, let $Y_n$ and $\bar Y_n$ denote the results of $n$ steps of the splitting integrator for the rank-$r$ approximation and rank-$q$ approximation, respectively, applied with step size $\h$.  Then, as long as $t_0+n\h\le T$,
$$
\| Y_n - \bar Y_n \| \le C(\eps+\h),
$$
where $C$ is independent of $n$, $\h$ and $\eps$ (but depends on $T-t_0$).
\end{theorem}

We note that by the standard error estimates of splitting methods, the integration error of the rank-$q$ approximation is $\bar Y_n - \bar Y(t_n) = O(h^p)$, uniformly in $\eps$, for the integrator of order~$p$. Furthermore, it follows from the over-approximation lemma in \cite{KocL07}, Section 5.3, that the difference of the rank-$r$ and rank-$q$ approximations is bounded by $Y(t)-\bar Y(t) = O(\eps)$.

\begin{proof}
(a)
We factorize
$$
Y_0=U_0S_0V_0^\top,
$$
where $U_0=(U_0',U_0'')\in \R^{m\times r}=\R^{m\times q}\times \R^{m\times(r-q)}$
and $V_0=(V_0',V_0'')\in \R^{n\times r}=\R^{n\times q}\times \R^{n\times(r-q)}$ have orthonormal columns.
 $S_0$ is chosen as an $r\times r$ matrix in block-diagonal form
$$
S_0 = \begin{pmatrix}
S_0' & 0 \\ 0 & S_0''
\end{pmatrix}
\quad\hbox{ with }\ S_0'=S'(t_0) \hbox{ and } S_0''=O(\eps).
$$
We consider the differential equation in the first splitting step, 
$\dt Y_I (t)= \dt A(t) P_{\calR(Y_I^\top(t))}$. We factorize (omitting the subscript $I$ and the argument $t$)
$$
Y = USV^\top, 
$$
where $U$ and $V$ have $r$ orthonormal columns, so that we have the equation
\begin{equation}\label{USV-I}
\dt USV^\top + U\dt SV^\top + US\dt V^\top = \dt A VV^\top.
\end{equation}
As is shown in  the proof of Lemma 5.4 of \cite{KocL07} (see also \cite{DieE99}), the decomposition becomes unique if we impose that $S$ stays block-diagonal,
$$
S= \begin{pmatrix}
S' & 0 \\ 0 & S''
\end{pmatrix},
$$
and 
$$
U^\top\dt U = H, \quad V^\top \dt V =K
$$
with $r\times r$ matrices of the block form
$$
H = \begin{pmatrix} 0 & H_{12} \\ H_{21} & 0\end{pmatrix}, \qquad
K = \begin{pmatrix} 0 & K_{12} \\ K_{21} & 0\end{pmatrix},
$$
which are skew-symmetric, $H_{12}=-H_{21}^T$ and $K_{12}=-K_{21}^T$.
With the corresponding decompositions $U=(U',U'')$ and $V=(V',V'')$ the proof of Lemma 5.4 of \cite{KocL07} yields that
$$
H_{12}=  -(S')^{-\top}V'^\top \dt A^\top U'' + O(\eps),\quad
K_{12}=  -(S')^{-1}U'^\top \dt A V'' + O(\eps),
$$
and in particular $H_{12}=O(1)$, $K_{12}=O(1)$.
Multiplying (\ref{USV-I}) with $U^\top$ from the left and with $V$ from the right, we obtain
$$
HS+\dt S+ SK^\top = U^\top \dt A V,
$$
which yields on the diagonal blocks
$$
\dt S' = U'^\top \dt A V', \quad \dt S''= U''^\top \dt A V''.
$$
From (\ref{USV-I}) we further obtain
$$
(U'S')^{\dt{\phantom{.}}}+ U'' S'' K_{12}^\top = \dt A V'
$$
and, using also the above equation for $\dt S'$,
$$
S'\dt V'^\top + H_{12} S'' V''^\top = U'^\top \dt A V'' V''^\top.
$$
We show in part (b) of the proof below that $U''^\top \dt A V''=O(\eps+\h)$ 
 (but we cannot conclude the same for $U'^\top \dt A V''$). 
In summary, we get (indicating now the subscript~$I$)
\begin{equation*}
\begin{split}
&\dt S_I' = U_I'^\top \dt A V_I', \qquad \dt S_I''=O(\eps+h)
\\
&
(U_I'S_I')^{\dt{\phantom{.}}} = \dt A V_I' + O(\eps+h)
\\
&S_I'\dt V_I'^\top  = U_I'^\top \dt A V_I'' V_I''^\top + O(\eps+h).
\end{split}
\end{equation*}
For the second step in the splitting we obtain similarly
\begin{equation*}
\begin{split}
&\dt S_{II}' = -U_{II}'^\top \dt A V_{II}', \qquad \dt S_{II}''=O(\eps+h)
\\
&
\dt U_{II}' S_{II}'= - U_{II}'' U_{II}''^\top \dt A V_{II}' + O(\eps+h)
\\
&S_{II}'\dt V_{II}'^\top  = - U_{II}'^\top \dt A V_{II}'' V_{II}''^\top + O(\eps+h),
\end{split}
\end{equation*}
and for the third step we obtain
\begin{equation*}
\begin{split}
&\dt S_{III}' = U_{III}'^\top \dt A V_{III}', \qquad \dt S_{III}''=O(\eps+h)
\\
&
\dt U_{III}' S_{III}'= U_{III}'' U_{III}''^\top \dt A V_{III}' + O(\eps+h)
\\
&(S_{III}' V_{III}^\top)^{\dt{\phantom{.}}}   = U_{III}'^\top \dt A  + O(\eps+h).
\end{split}
\end{equation*}

Comparing these equations with those for $\bar S,\bar U, \bar V$ in the rank-$q$ splitting method, it follows that 
\begin{equation*}
\begin{split}
&
S_1' = S_{III}'(t_1) = \bar S_1 + O(\h\eps + \h^2)
\\
&
U_1' = U_{III}'(t_1) = \bar U_1 + O(\h\eps + \h^2)
\\
&
V_1' = V_{III}'(t_1) = \bar V_1 + O(\h\eps + \h^2).
\end{split}
\end{equation*}
Since  stable error propagation in the rank-$q$ splitting method is obtained by the standard argument using the Lipschitz continuity of the right-hand side of the differential equation (uniformly in $\eps$), we conclude to the assertion of the theorem.

(b) We decompose the rank-$q$ matrix
$$
A'(t) = \hat U'(t) \hat S'(t) \hat V'(t)^\top \quad\ \hbox{ with}\quad
 \hat U'(t)  \dt{\hat U'}(t)=0,\  \hat V'(t)  \dt{\hat V'}(t)=0
$$
and $\hat U'(t_0)=U'_0+O(\eps+h)$, $\hat V'(t_0)=V_0'+O(\eps+h)$. This choice of initial factors is possible because of the condition $A'(t_0)=Y_0+O(\eps+h)$, and the factorization at later $t$ exists by solving the rank-$q$ differential equations (\ref{integr:system}). For $t_0\le t \le t_0 +\h$ we have $U'(t)= U'(t_0) + O(\h) = \hat U'(t_0) +O(\eps+\h)=\hat U'(t) +O(\eps+\h)$ and $V'(t)=  \hat V'(t) +O(\eps+\h)$ , and hence
\begin{equation*}
\begin{split}
\dt{A'}(t) = \dt{\hat U'}(t) \hat S'(t) V'(t)^\top + 
 U'(t) \dt{\hat S'}(t) V'(t)^\top +  U'(t) \hat S'(t) \dt{\hat V'}(t)^\top + O(\eps+\h).
\end{split}
\end{equation*}
Since $U''(t)^\top U'(t)=0$ and  $V''(t)^\top V'(t)=0$, this yields the desired bound
$$
U''(t)^\top \dt A(t) V''(t)=O(\eps+\h),
$$
and the proof is complete.
\end{proof}

\section{Numerical experiments}
\label{sec-numexp}
\subsection{Problem setting}
The example is taken from \cite{KocL07}. We generate 
time-dependent matrices $A(t)$ as
$$
     A(t) = Q_1(t) (A_1 + \exp(t) A_2) Q_2(t),
$$
where the orthogonal matrices $Q_1(t)$ and $Q_2(t)$ are generated as the
solutions of differential equations
$$
 \dt{Q}_i = T_i Q_i, \quad i=1,2
$$
with random skew-symmetric matrices $T_i$. The matrices $A_1, A_2$ are
generated as follows. First, we generate a $10 \times 10$ matrix as an
identity matrix plus a matrix with random entries distributed
uniformly over $[0,0.5]$. This matrix is then set as a leading block
of a $100 \times 100$ matrix. After that, we add a perturbation to
this enlarged matrix. The perturbation is generated as a matrix with uniformly
distributed entries on $[0,\eps]$.  For small $\eps$ this matrix is close to a rank-$10$ matrix. If the rank of the approximation is chosen larger than $10$, the norm of $S^{-1}$ in (\refeq{integr:system})
will be
large and this leads to instability with standard time discretizations.
\subsection{Numerical comparisons}
We will compare the following schemes. The first scheme is the
standard implicit midpoint rule combined with fixed point iteration
applied directly to the system
\eqref{integr:system}. 
We test the proposed
splitting schemes with different orders of splitting and with/without
symmetrization. We denote by KSL the scheme of Section 3.2, where first $K=US$ is updated, then $S$, then $L=VS^\top$. By KLS we denote the scheme where first $K$, then $L$, then $S$ are updated. We thus consider the following methods:
\begin{enumerate}
\item KLS scheme (first order)
\item KLS scheme with symmetrization (second order)
\item KSL scheme (first order)
\item KSL scheme with symmetrization  (second order)
\end{enumerate}
We are interested in the approximation errors $\|Y(t) - A(t)\|$ for
each particular scheme and different values of $r$ and
$\eps$. The results are presented in
Figure~\ref{ksl:figure1}, where we also plot the error of the best rank-$r$ approximation to $A(t)$ computed by SVD. Note that both KSL schemes perform
remarkably better in the overapproximation case (subplot d). The
midpoint rule is unstable in case d), whereas both KLS schemes have
significantly higher error than the KSL schemes. All computations are done with constant step size $h=10^{-3}$.
\begin{figure}[!htbp]
\begin{center}
\resizebox{12cm}{!}
{
\input{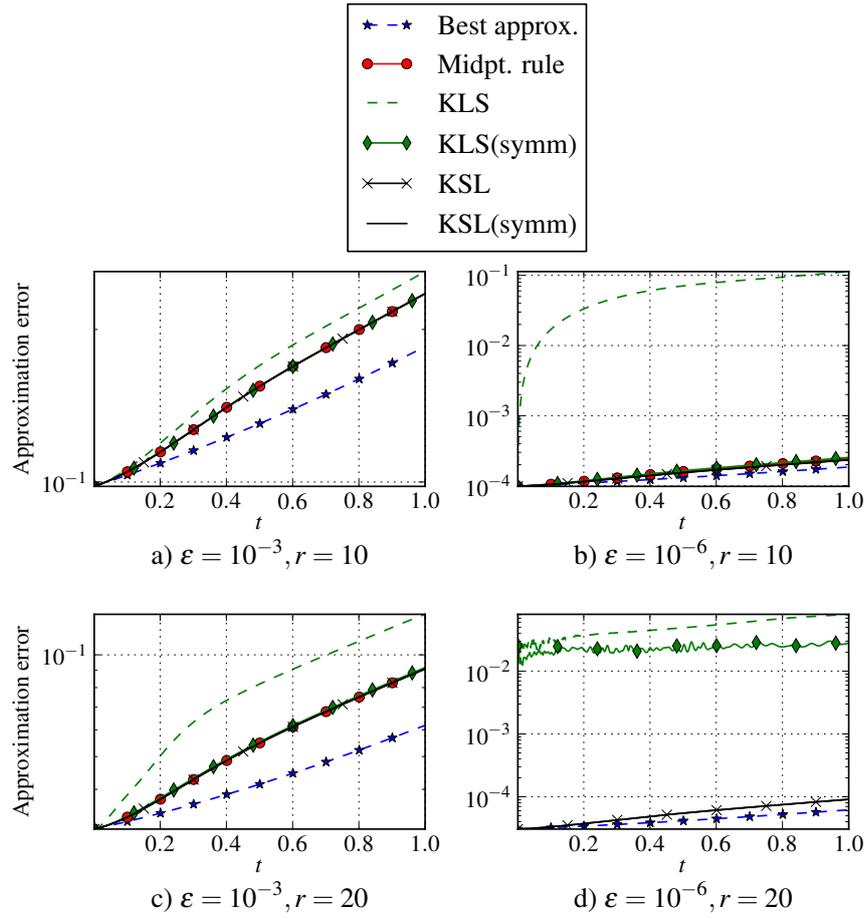}
}
\end{center}
\caption{Dynamical low-rank approximation error for different schemes
  and different values of approximation rank $r$ and perturbation size $\eps$.}\label{ksl:figure1}
\end{figure}

To test the convergence properties of the schemes with respect to the
timestep $h$, we have computed the numerical order of the different
schemes using the Runge rule:
$$
 \| y(h) - y(h/2) \| \,/ \, \|y(h/2)-y(h/4)\| \approx 2^p \quad\hbox{ in case of order $p$.} 
$$

\begin{table}[!htbp]
\parbox{0.45\linewidth}{
    \centering
\begin{tabular}{llll}
\hline\noalign{\smallskip}
    & $p$ & Appr. err. \\    
\noalign{\smallskip}\hline\noalign{\smallskip}
Midpoint  & 2.0023 &   0.2200  \\
KLS       & 1.0307  &  1.8133 \\
KLS(symm) & 1.8226 &   0.2215 \\
KSL       & 1.0089 &   0.2188 \\
KSL(symm) & 2.005  &   0.2195 \\ 
\noalign{\smallskip}\hline
\end{tabular}
\caption{$\eps=10^{-3}, r=10$}
}
\hfill
\parbox{0.45\linewidth}{
\begin{tabular}{llll}
\hline\noalign{\smallskip}
    & $p$ & Appr. err. \\    
\noalign{\smallskip}\hline\noalign{\smallskip}
Midpoint  & 2.0024   & 0.0188  \\
KLS       & 1.0309   & 1.8030 \\
KLS(symm) & 1.8231   & 0.0324 \\
KSL       & 1.0082   & 0.0002  \\
KSL(symm) & 2.0049   & 0.0002 \\
\noalign{\smallskip}\hline
\end{tabular}
\caption{$\eps=10^{-6}, r=10$}
}
\end{table}
\begin{table}[!htbp]
\parbox{.45\linewidth}{
    \centering
\begin{tabular}{llll}
\hline\noalign{\smallskip}
    & $p$ & Appr. err. \\    
\noalign{\smallskip}\hline\noalign{\smallskip}
Midpoint & 0.0001   & 0.1006 \\
KLS      & 0.8154   & 1.4224 \\
KLS(symm)& 1.4911   & 0.3142 \\
KSL      & 1.0354   & 0.0913 \\
KSL(symm)& 1.9929   & 0.0913  \\
\noalign{\smallskip}\hline
\end{tabular}
\caption{$\eps=10^{-3}, r=20$}
}
\hfill
\parbox{.45\linewidth}{
\begin{tabular}{lccc}
\hline\noalign{\smallskip}
    & $p$ & Appr. err. \\    
\noalign{\smallskip}\hline\noalign{\smallskip}
Midpoint  & -               & failed  \\
KLS       & 0.9633          & 1.3435 \\
KLS(symm) & 0.3127          & 1.5479 \\
KSL       & 1.0362          & 9.1316e-05  \\
KSL(symm) & 1.993           & 9.1283e-05\\
\noalign{\smallskip}\hline
\end{tabular}
\caption{$\eps=10^{-6}, r=20$}
}
\end{table}

The approximation error listed is the error at $t=1$ with respect to the given matrix $A(t)$, not with respect to the solution $Y(t)$ of rank $r$ of the differential equations \eqref{integr:system}.
In the overapproximation case both KLS schemes perform much
worse, and the symmetrized version loses its second order. 
The KSL scheme, that is, the scheme of Section~3.2, and its symmetrized version (see Section 3.3) clearly outperform the other methods.

The last test describes the stability of the schemes with respect to the time step $h$. In Figure~\ref{ksl:test-tau} we plot the
approximation error at time $t=1$ for both KSL schemes and the midpoint rule for different $h$ ranging from $10^{-1}$ to $10^{-3}$. The
rank $r$ was set to $20$ (overapproximation case) and the noise level was chosen to be $\eps = 10^{-3}$.
\begin{figure}[!htbp]
\begin{center}
\resizebox{10cm}{!}
{
    \input{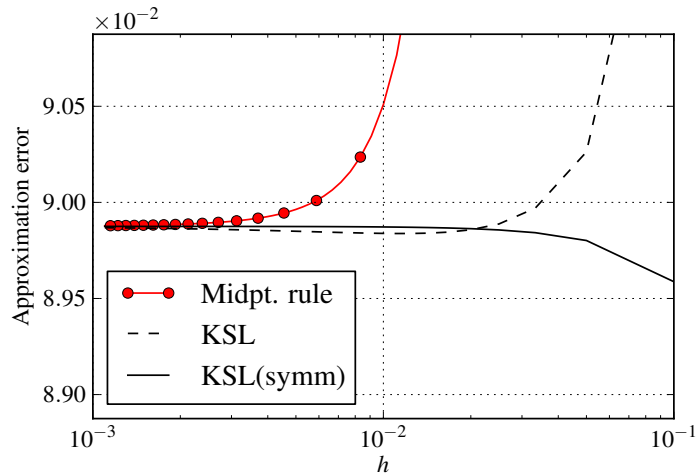}
}
\end{center}
\caption{Dynamical low-rank approximation error for different schemes
  versus stepsize $h$, for $r=20$ and $\eps=10^{-3}$}\label{ksl:test-tau}
\end{figure}
The midpoint rule becomes unstable for large values of $h$, whereas 
both KSL schemes give a good result for the whole range of stepsizes. It is interesting to compare this value with the best 
low-rank approximation of $A(1)$. The error of the best rank-$10$ approximation of $A(1)$ (computed by the SVD) is approximately $10^{-2}$ 
and the error of the best rank-$20$ approximation is approximately $3 \cdot 10^{-3}$. The dynamical low-rank approximation thus captures the ``smooth'' component of the solution.   

\section{Conclusion and perspectives}

We have proposed and analyzed a fully explicit, computationally inexpensive integrator for the dynamical low-rank approximation of time-dependent matrices that is based on splitting the projector onto the tangent space of the low-rank manifold. The integrator has remarkable robustness under over-approximation with a too high rank. While standard explicit and implicit integrators break down in such a situation, the integrator proposed here does not suffer from the ill-conditioning or singularity of the small matrix factor in the orthogonal low-rank factorization.

The robustness under
overapproximation enables one to control the rank adaptively. Lowering
the rank is trivial, but thanks to the robustness under a reduced rank we are able to raise the rank in a
natural way, continuing the computations with the higher rank starting from
the values of lower rank. This is  not possible with
standard integrators applied to the higher-rank differential
equations, which would have to start from a singular matrix factor, in which case the differential equations are not well-defined.

Another application of the proposed integrator is in optimization algorithms on a low-rank manifold, such as cg or Newton's method. There, an update $A+\Delta A$ to a low-rank iterate $A$ has to be truncated (or retracted in another terminology) back to the given low rank. This can be done efficiently by applying one step of our integrator to $A+t\Delta A$ at $t=1$.

The integrator can be extended  to the low-rank approximation of tensors in the tensor train and hierarchical Tucker formats. This extension will be studied in forthcoming work.

\end{document}